\documentclass[11pt]{article}
\usepackage{latexsym,amssymb,amsmath,amsthm,enumerate,float,geometry,cite}
\newtheorem{theorem}{Theorem}
\newtheorem{lemma}[theorem]{Lemma}
\newtheorem{proposition}[theorem]{Proposition}

\newtheorem{conjecture}[theorem]{Conjecture}

\newtheorem{claim}{Claim}

\usepackage{lineno}
\usepackage{setspace}

\allowdisplaybreaks

\begin{document}
\onehalfspace

\title{Efficiently finding low-sum copies of spanning forests in\\ 
zero-sum complete graphs via conditional expectation}
\author{Johannes Pardey \and Dieter Rautenbach}
\date{}

\maketitle
\vspace{-10mm}
\begin{center}
{\small Institute of Optimization and Operations Research, Ulm University,\\ 
Ulm, Germany, \texttt{$\{$johannes.pardey,dieter.rautenbach$\}$@uni-ulm.de}}
\end{center}

\begin{abstract}
For a fixed positive $\epsilon$,
we show the existence of a constant $C_\epsilon$ with the following property:
Given a $\pm1$-edge-labeling $c:E(K_n)\to \{ -1,1\}$ 
of the complete graph $K_n$ with $c(E(K_n))=0$,
and a spanning forest $F$ of $K_n$ of maximum degree $\Delta$,
one can determine in polynomial time 
an isomorphic copy $F'$ of $F$ in $K_n$ with 
$|c(E(F'))|\leq \left(\frac{3}{4}+\epsilon\right)\Delta+C_\epsilon.$
Our approach is based on the method of conditional expectation.\\[3mm]
\end{abstract}

\section{Introduction}

The kind of zero-sum problem that we study here 
can be traced back to algebraic results 
such as the well-known Erd\H{o}s-Ginzburg-Ziv theorem~\cite{ergizi}.
The two survey articles due to Caro~\cite{ca} as well as Gao and Geroldinger~\cite{gage} 
give a detailed account of this area also known as zero-sum Ramsey theory 
within discrete mathematics and additive group theory.

Several recent papers~\cite{cahalaza,cayu,ehmora,kisi,mopara} study 
(almost) zero-sum spanning forests in edge-labeled complete graphs,
and, in the present paper, 
we contribute an algorithmic approach for finding low-sum spanning forests.
The setting involves a complete graph $K_n$ of order $n$ 
together with a {\it zero-sum $\pm 1$-labeling} of its edges, 
that is, a function $c:E(K_n)\to \{ -1,1\}$ that satisfies
$$c(E(K_n))=\sum\limits_{e\in E(K_n)}c(e)=0.$$
For a given spanning forest $F$ of $K_n$, 
we consider the algorithmic task to efficiently find an isomorphic copy $F'$ of $F$ in $K_n$ 
that minimizes $|c(E(F'))|$.
The corresponding existence version of this algorithmic task,
and, in particular, the question under which conditions 
there is a zero-sum copy of $F$ in $K_n$ 
was studied in~\cite{cahalaza,cayu,ehmora,kisi,mopara},
where the arguments are typically non-algorithmic 
or do not lead to efficient algorithms.

The following two simple observations correspond 
to key existential arguments in this area:
\begin{itemize}
\item If $F$ is a spanning forest of $K_n$ and $c$ is a zero-sum $\pm 1$-labeling of the edges of $K_n$,
then, by symmetry, every edge of $K_n$ belongs to the same number of isomorphic copies of $F$ in $K_n$,
which implies that the average of $c(E(F'))$, 
where $F'$ ranges over all isomorphic copies of $F$ in $K_n$, equals $0$.
In particular, there are copies $F^+$ and $F^-$ of $F$ with $c(E(F^+))\geq 0$ and $c(E(F^-))\leq 0$.
\item If $F_1,\ldots,F_k$ are isomorphic copies of $F$ in $K_n$,
$c(E(F_1))\geq 0$, 
$c(E(F_k))\leq 0$, and
each $F_{i+1}$ arises from $F_i$ by removing at most $\ell$ edges and adding at most $\ell$ edges,
then  
$|c(E(F_i))|\leq \ell$
for some $i$.
\end{itemize}
As observed in \cite{mopara}, these observations yield the following.

\begin{proposition}[Mohr et al.~\cite{mopara}]\label{proposition1}
If $c:E(K_n)\to \{ -1,1\}$ is a zero-sum labeling of $K_n$,
and $F$ is a spanning forest of $K_n$ of maximum degree $\Delta$,
then there is an isomorphic copy $F'$ of $F$ in $K_n$ with 
$|c(E(F'))|\leq \Delta+1$.
\end{proposition}
Believing that the bound in Proposition \ref{proposition1} is not best-possible,
Mohr et al.~\cite{mopara} posed the following.
\begin{conjecture}[Mohr et al.~\cite{mopara}]\label{conjecture2}
If $c:E(K_n)\to \{ -1,1\}$ is a zero-sum labeling of $K_n$,
and $F$ is a spanning forest of $K_n$ of maximum degree $\Delta$,
then there is an isomorphic copy $F'$ of $F$ in $K_n$ with 
$|c(E(F'))|\leq \frac{1}{2}\Delta-\frac{1}{2}$.
\end{conjecture}
In~\cite{mopara} this conjecture was verified for the spanning star $K_{1,n-1}$ of $K_n$.
For other special spanning forests,
in particular, perfect matchings,
and under natural divisibility conditions,
the existence of zero-sum copies was shown in~\cite{cahalaza,ehmora,kisi,mopara}.
Our main contribution here is the following theorem, 
which improves Proposition \ref{proposition1} in two ways:
It strengthens the bound given there almost halfway towards the bound from Conjecture \ref{conjecture2},
and it provides the existence of an efficient algorithm to find the desired low-sum copy.

\begin{theorem}\label{theorem1}
Let $\epsilon>0$ be fixed.
There is a constant $C_\epsilon$ such that the following holds:
Given a zero-sum labeling $c:E(K_n)\to \{ -1,1\}$ of $K_n$
and a spanning forest $F$ of $K_n$ of maximum degree $\Delta$,
one can determine in polynomial time 
an isomorphic copy $F'$ of $F$ in $K_n$ with 
$$|c(E(F'))|\leq \left(\frac{3}{4}+\epsilon\right)\Delta+C_\epsilon.$$
\end{theorem}
Our approach is based on the {\it method of conditional expectation} \cite{alsp,erse}.
In Section \ref{section2} we explain how to implement this method in the present context,
and illustrate it with an algorithmic version of Proposition \ref{proposition1}.
In Section \ref{section3} we consider a natural greedy algorithm based on the method of conditional expectation,
and provide the proof of Theorem \ref{theorem1} by analyzing this greedy algorithm.
In a conclusion we discuss further possible developments. 

\section{Embedding $T$ via conditional expectation}\label{section2}

Throughout this section, let $K$ be a complete graph of order $n$,
let $c:E(K)\to \{ -1,1\}$ be a zero-sum labeling of $K$, and 
let $F$ be a spanning forest of $K$ of maximum degree $\Delta$.
Let $[n]$ be the set of positive integers at most $n$.
We may assume that $K$ has vertex set $[n]$.
For a subgraph $H$ of $K$, 
let $c(H)$ abbreviate $c(E(H))$, and
let $\bar{c}(H)$ equal $\frac{c(H)}{|E(H)|}$.
Similarly, for a set $E$ of edges of $K$, 
let $\bar{c}(E)$ equal $\frac{c(E)}{|E|}$.
As usual, for a set $X$ of vertices of a graph $G$, 
let $G[X]$ be the subgraph of $G$ induced by $X$,
and let $G-X=G[V(G)\setminus X]$.

For a permutation $\pi$ in $S_n$, 
let $F_\pi$ be the isomorphic copy of $F$ within $K$ 
with edge set $\{ \pi(u)\pi(v):uv\in E(F)\}$,
that is, within $F_\pi$, 
the vertex $\pi(u)$ of $K$ assumes the role of the vertex $u$ of $F$.
The first of the two observations mentioned before Proposition \ref{proposition1}
can be expressed as follows:
Choosing a permutation $\pi$ from $S_n$ uniformly at random,
and considering the random variable $c(F_\pi)$, 
linearity of expectation implies that
\begin{eqnarray}\label{ee-1}
\mathbb{E}\left[c(F_\pi)\right]=\frac{1}{n!}\sum\limits_{\pi\in S_n}c(F_\pi)=\bar{c}(K)m(F)=0.
\end{eqnarray}
For $k$ in $[n]$, and distinct elements $i_1,\ldots,i_k$ of $[n]$,
we consider the expected value of the random variable $c(F_\pi)$
subject to the condition that $\pi(j)=i_j$ for every $j$ in $[k]$, that is,
\begin{eqnarray}\label{ee0}
\mathbb{E}\left[c(F_\pi)\mid (\pi(1),\ldots,\pi(k))=(i_1,\ldots,i_k)\right]
=\frac{1}{(n-k)!}\sum\limits_{\pi\in S_n:(\pi(1),\ldots,\pi(k))=(i_1,\ldots,i_k)}c(F_\pi).
\end{eqnarray}
By the uniformity of the choice of the random permutation $\pi$, we have
\begin{eqnarray}\label{ee1}
\mathbb{E}\left[c(F_\pi)\right] &=& \frac{1}{n}\sum_{i_1=1}^n\mathbb{E}\left[c(F_\pi)\mid \pi(1)=i_1\right].
\end{eqnarray}
Similarly, if $k<n$, 
then the uniformity of the choice of $\pi$ implies that the conditional expectation
(\ref{ee0}) equals
\begin{eqnarray}
\frac{1}{n-k}\sum_{i_{k+1}\in [n]\setminus \{ i_1,\ldots,i_k\}}
\mathbb{E}\left[c(F_\pi)\mid (\pi(1),\ldots,\pi(k+1))=(i_1,\ldots,i_{k+1})\right].\label{ee2}
\end{eqnarray}
For the approach, 
it is crucial that 
the conditional expectation (\ref{ee0}) can be calculated efficiently.
In fact, linearity of expectation implies that (\ref{ee0}) equals
\begin{eqnarray}\label{ee3}
&&c\left(F_\pi\left[\{  i_1,\ldots,i_k\}\right]\right)\label{ee3}\\
&+&\sum\limits_{j=1}^k
\bar{c}\left(\{ i_j\ell:\ell\in [n]\setminus \{  i_1,\ldots,i_k\}\}\right)|N_F(j)\setminus [k]|\label{ee4}\\
&+&\bar{c}\left(K-\{  i_1,\ldots,i_k\}\right) m(F-[k]),\label{ee5}
\end{eqnarray}
where 
\begin{itemize}
\item (\ref{ee3}) is the weight of the edges $uv$ of $F[k]$ 
whose embedding $\pi(u)\pi(v)$ within $K$ is already completely determined 
by the condition $(\pi(1),\ldots,\pi(k))=(i_1,\ldots,i_k)$,
\item (\ref{ee4}) collects the expected weights of the $|N_F(j)\setminus [k]|$ edges
that each vertex $i_j$ of $F_\pi$,
taking the role of the vertex $j$ of $F$, 
sends into the set $[n]\setminus \{  i_1,\ldots,i_k\}$, and
\item (\ref{ee5}) is the expected weight of the edges of $F-[k]$
whose embedding within $K-\{  i_1,\ldots,i_k\}$ 
is still chosen uniformly at random.
\end{itemize}
In view of this representation, we obtain the following.

\begin{lemma}\label{lemma1}
Given $K$, $c$, $F$, $k$, and $i_1,\ldots,i_k$,
$$\mathbb{E}[c(F_\pi)\mid (\pi(1),\ldots,\pi(k))=(i_1,\ldots,i_k)]$$
can be computed in polynomial time.
\end{lemma}
Implementing the method of conditional expectation in the present context,
we will now explain how to determine a permutation $\pi$ in $S_n$,
in other words, an isomorphic copy $F_\pi$ of $F$ within $K$,
for which $|c(F_\pi)|$ is small
by fixing the values $\pi(1),\ldots,\pi(n)$ one by one. 

There are different reasonable ways to do this.

Mimicking the proof of Proposition \ref{proposition1} in \cite{mopara},
we obtain the following.

\begin{proposition}\label{proposition2}
Given $K$, $c$, and $F$,
a permutation $\pi$ in $S_n$ with $|c(F_\pi)|\leq \Delta+1$ 
can be determined in polynomial time.
\end{proposition}
\begin{proof}
By (\ref{ee1}) and (\ref{ee2}), there is an ordering $i_1,\ldots,i_n$ of $[n]$ such that
\begin{eqnarray*}
\mathbb{E}\left[c(F_\pi)\mid \pi(1)=i_1\right] & \geq & \mathbb{E}\left[c(F_\pi)\right]=0\mbox{ and}\\
\mathbb{E}[c(F_\pi)\mid (\pi(1),\ldots,\pi(k+1))=(i_1,\ldots,i_{k+1})] 
&\geq & \mathbb{E}[c(F_\pi)\mid (\pi(1),\ldots,\pi(k))=(i_1,\ldots,i_k)]
\end{eqnarray*}
for every $k$ in $[n-1]$.
Furthermore, by (\ref{ee1}), (\ref{ee2}), and Lemma \ref{lemma1},
such an ordering $i_1,\ldots,i_n$ can be found in polynomial time.
In other words, in polynomial time one can determine a
permutation $\pi_+$ in $S_n$ with $c(F_{\pi_+})\geq 0$.
Similarly, in polynomial time one can determine a
permutation $\pi_-$ in $S_n$ with $c(F_{\pi_-})\leq 0$.
Considering transpositions of pairs of vertices of $F$ 
always involving at least one vertex of degree at most $1$, 
cf.~the proof of Proposition \ref{proposition1} in~\cite{mopara},
one can determine in polynomial time 
a sequence $\pi_0,\ldots,\pi_r$
of permutations from $S_n$ such that 
$r$ is polynomially bounded in terms of $n$,
$\pi_0=\pi_+$, 
$\pi_r=\pi_-$, and,
for every $i$ in $[r]$,
$F_{\pi_i}$ arises from $F_{\pi_{i-1}}$
by removing at most $\Delta+1$ edges and adding at most $\Delta+1$ edges.
The second of the two observations mentioned before Proposition \ref{proposition1}
implies $\min\left\{ |c(F_{\pi_i})|:i\in \{ 0,\ldots,r\}\right\}\leq \Delta+1$,
and returning a permutation $\pi_i$ minimizing $|c(F_{\pi_i})|$
accomplishes the desired task.
\end{proof}
Proposition \ref{proposition2} does not really exploit that $F$ is a forest.
In fact, it can easily be adapted to the situation in which $F$
is not a forest 
replacing the bound $\Delta+1$ 
by $\Delta+\delta$,
where $\delta$ is the minimum degree of $F$.
While Proposition \ref{proposition2} corresponds to an algorithmic version 
of the existential argument behind Proposition \ref{proposition1},
there is actually a more natural way of implementing 
the method of conditional expectation for our problem,
choosing the vertices $i_1,\ldots,i_n$ 
one by one in this order
in such a way that the absolute value of each conditional expected value
(\ref{ee0}) is as small as possible.
Our proof of Theorem \ref{theorem1} 
relies on the analysis of this more natural greedy approach.
Note that there is one degree of freedom that we did not exploit so far;
we can freely choose the order in which the vertices of $F$ 
are embedded one by one into $K$.

\section{Proof of Theorem \ref{theorem1}}\label{section3}

Throughout this section, 
let $\epsilon>0$,
let $K$ be a complete graph of order $n$,
let $c:E(K)\to \{ -1,1\}$ be a zero-sum labeling of $K$, and 
let $F$ be a spanning forest of $K$ of maximum degree $\Delta$.
In view of the statement of Theorem \ref{theorem1} and Proposition \ref{proposition1}, 
we may assume that $\epsilon<\frac{1}{4}$,
and that $n$ is sufficiently large in terms of $\epsilon$.
Possibly replacing $\epsilon$ by a slightly smaller value, 
we may furthermore assume, for notational simplicity, that $\epsilon n$ is an integer.

Since $m(F)\leq n-1$, the forest $F$ has less than $\epsilon n$ 
vertices of degree more than $\frac{2}{\epsilon}$.
Hence, since every induced subgraph of $F$ is $1$-degenerate,
we may assume, possibly by reordering/renaming the vertices of $K$ and $F$, that
\begin{eqnarray}
|N_F(i)\cap [i-1]|&\leq &1\,\,\,\,\,\,\,\mbox{ for every $i$ in $[\epsilon n]$, and}\label{ec1}\\
d_F(i)&\leq &\frac{2}{\epsilon}\,\,\,\,\,\,\,\mbox{ for every $i$ in $[n]\setminus [\epsilon n]$.}\label{ec2}
\end{eqnarray}
Note that the possible reordering/renaming of the vertices of $K$
can be performed in polynomial time.

For distinct $i_1,\ldots,i_k$ from $[n]$, let
$$\mathbb{E}[i_1,\ldots,i_k]=\mathbb{E}[c(F_\pi)\mid (\pi(1),\ldots,\pi(k))=(i_1,\ldots,i_k)].$$
For $k=0$, let $\mathbb{E}[i_1,\ldots,i_k]=\mathbb{E}[c(F_\pi)]\stackrel{(\ref{ee-1})}{=}0$.

Now, in order to determine $F'$ in polynomial time, 
we consider the following natural greedy algorithm:
\begin{quote}
{\it Choose $i_1,\ldots,i_n$
one by one in this order in such a way that in every step
$\left|\mathbb{E}[i_1,\ldots,i_k]\right|$
is minimized,
that is,
$i_j={\rm arg}\min \{ \left|\mathbb{E}[i_1,\ldots,i_k,p]\right|:p\in [n]\setminus \{ i_1,\ldots,i_k\}\}$ for every $j$ in $[n]$.}
\end{quote}
By (\ref{ee1}), (\ref{ee2}), and Lemma \ref{lemma1},
the values $i_1,\ldots,i_n$,
which completely determine $\pi$ and $F'$,
can be determined in polynomial time.
Therefore, in order to complete the proof,
it suffices to show that  
\begin{eqnarray}\label{ec3}
\left|\mathbb{E}[i_1,\ldots,i_k]\right|
& \leq & 
\left(\frac{3}{4}+327\epsilon\right)\Delta+\left(\frac{8}{\epsilon}+4\right)
\end{eqnarray}
for every $k$ in $[n]$.
We establish (\ref{ec3}) using the following two claims.

\begin{claim}\label{claim1}
For every $k$ in $\{ 0,\ldots,n-1\}$,
there is some $p$ in $[n]\setminus \{ i_1,\ldots,i_k\}$
such that
\begin{itemize}
\item If $\mathbb{E}[i_1,\ldots,i_k]>0$, then $\mathbb{E}[i_1,\ldots,i_k,p]\leq \mathbb{E}[i_1,\ldots,i_k]$,
\item if $\mathbb{E}[i_1,\ldots,i_k]<0$, then $\mathbb{E}[i_1,\ldots,i_k,p]\geq \mathbb{E}[i_1,\ldots,i_k]$, and
\item 
$$\left|\mathbb{E}[i_1,\ldots,i_k,p]-\mathbb{E}[i_1,\ldots,i_k]\right|
\leq \left(1+\frac{16}{3}\epsilon\right)\Delta+\left(\frac{8}{\epsilon}+4\right).$$
\end{itemize}
\end{claim}

\begin{claim}\label{claim2}
For every $k$ in $\{ 0,\ldots,n-1\}$, 
there is some $p$ in $[n]\setminus \{ i_1,\ldots,i_k\}$
such that 
$$\left|\mathbb{E}[i_1,\ldots,i_k,p]-\mathbb{E}[i_1,\ldots,i_k]\right|
\leq 
\left(\frac{1}{2}+327\epsilon\right)\Delta+\left(\frac{8}{\epsilon}+4\right).$$
\end{claim}
Before we prove these two claims, we explain how they imply (\ref{ec3}).
Since $c$ is a zero-sum labeling, 
(\ref{ec3}) holds for $k=0$.
Now, 
if (\ref{ec3}) holds for some $k$ 
and $|\mathbb{E}[i_1,\ldots,i_k]|\leq \frac{1}{4}\Delta$,
then Claim \ref{claim2} implies the existence of a possible choice $p$ for $i_{k+1}$
with $|\mathbb{E}[i_1,\ldots,i_{k+1}]|$
bounded as in (\ref{ec3}).
Therefore, by the selection rule of the greedy algorithm,
(\ref{ec3}) holds for $k+1$ (instead of $k$).
Otherwise,
if (\ref{ec3}) holds for some $k$ but $|\mathbb{E}[i_1,\ldots,i_k]|>\frac{1}{4}\Delta$,
then Claim \ref{claim1} implies the existence of a possible choice $p$ for $i_{k+1}$
with $|\mathbb{E}[i_1,\ldots,i_{k+1}]|$
bounded as in (\ref{ec3}).
Again, also in this case, (\ref{ec3}) holds for $k+1$ (instead of $k$).
Altogether, a simple inductive argument yields (\ref{ec3}) for all $k$ in $[n]$.

We fix some abbreviating notation.

For $k\in [n]$ and $j\in [k]$, let 
\begin{eqnarray*}
c_{[k]}&=&c\left(F_\pi\left[\{  i_1,\ldots,i_k\}\right]\right),\\
d_k(j) & = & |N_F(j)\setminus [k]|,\\
m_k & = & m(F-[k]),\\
\bar{c}_k(j) &=& \bar{c}\left(\{ i_j\ell:\ell\in [n]\setminus \{  i_1,\ldots,i_k\}\}\right),\mbox{ and }\\
\bar{c}_k &=& \bar{c}\left(K-\{  i_1,\ldots,i_k\}\right).
\end{eqnarray*}
With these abbreviations,
\begin{eqnarray}\label{ee0b}
\mathbb{E}[i_1,\ldots,i_k]
=c_{[k]}
+\sum\limits_{j=1}^k\bar{c}_k(j)d_k(j)
+\bar{c}_km_k.
\end{eqnarray}

\begin{claim}\label{claim3}
If $p$ and $q$ are positive integers with $q<p$ and $x_1,\ldots,x_p\in \{ -1,1\}$,
then
$$
\left|\frac{1}{p}\sum_{i=1}^px_i-\frac{1}{p-q}\sum_{i=1}^{p-q}x_i\right|
\leq \frac{2q}{p}.
$$
\end{claim}
\begin{proof}
\begin{eqnarray*}
\left|\frac{1}{p}\sum_{i=1}^px_i-\frac{1}{p-q}\sum_{i=1}^{p-q}x_i\right|
& = & \left|\left(\frac{1}{p}-\frac{1}{p-q}\right)\sum_{i=1}^{p-q}x_i+\frac{1}{p}\sum_{i=p-q+1}^{p}x_i\right|
\leq 
\frac{q}{p(p-q)}\underbrace{\left|\sum_{i=1}^{p-q}x_i\right|}_{\leq p-q}
+\frac{1}{p}\underbrace{\left|\sum_{i=p-q+1}^{p}x_i\right|}_{\leq q}
\leq \frac{2q}{p}.
\end{eqnarray*}
\end{proof}
Below, we shall apply Claim \ref{claim3} mainly in the following settings:
\begin{itemize}
\item $p=n-k$ and $q=1$, in which case $\frac{2q}{p}=\frac{2}{n-k}$, and
\item $p={n-k\choose 2}$ and $q=n-k-1$, in which case $\frac{2q}{p}=\frac{4}{n-k}$.
\end{itemize}
We proceed to the proof of Claim \ref{claim1}.

\begin{proof}[Proof of Claim \ref{claim1}]
By symmetry, we may assume that $\mathbb{E}[i_1,\ldots,i_k]>0$.
By (\ref{ee2}), there is some $p$ in $[n]\setminus \{ i_1,\ldots,i_k\}$ with 
$\mathbb{E}[i_1,\ldots,i_k,p]\leq \mathbb{E}[i_1,\ldots,i_k]$.
We will argue that $p$ already satisfies the desired inequality.

Let 
$$d_1=|N_F(k+1)\cap [k]|
\,\,\,\,\,\, \mbox{ and }\,\,\,\,\,\, d_2=d_F(k+1)-d_1,$$
that is, $d_2=d_{k+1}(k+1)=|N_F(k+1)\setminus [k+1]|$.

Let
\begin{eqnarray*}
c_1 & = & c(\{ pi_\ell:\ell\in N_F(k+1)\cap [k]\}),\\
\bar{c}'_k(j) &=& \bar{c}\left(\{ i_j\ell:\ell\in [n]\setminus \{  i_1,\ldots,i_k,p\}\}\right)\mbox{ for $j$ in $[k]$},\\
\bar{c}'_k(p) &=& \bar{c}\left(\{ p\ell:\ell\in [n]\setminus \{  i_1,\ldots,i_k,p\}\}\right),\mbox{ and}\\
\bar{c}'_k &=& \bar{c}\left(K-\{  i_1,\ldots,i_k,p\}\right),
\end{eqnarray*}
that is, going 
from $\bar{c}_k(j)$ to $\bar{c}'_k(j)$ or 
from $\bar{c}_k$ to $\bar{c}'_k$ 
corresponds to the possibly alternative choice of $i_{k+1}$ as $p$.

Note that 
$m_{k+1}=m_k-d_2$
and $|c_1|\leq d_1$.

Furthermore, by Claim \ref{claim3},
$|\bar{c}'_k(j)-\bar{c}_k(j)|\leq \frac{2}{n-k}$ for $j\in [k]$
and
$|\bar{c}'_k-\bar{c}_k|\leq \frac{4}{n-k}$.

By (\ref{ee0b}), we have
\begin{eqnarray}
&& |\mathbb{E}[i_1,\ldots,i_k,p]-\mathbb{E}[i_1,\ldots,i_k]|\nonumber\\
&=& \left|c_1
+\sum\limits_{j=1}^k\bar{c}'_k(j)d_{k+1}(j)
+\bar{c}'_k(p)d_2
+\bar{c}'_km_{k+1}
-\sum\limits_{j=1}^k\bar{c}_k(j)d_k(j)
-\bar{c}_km_k
\right|\nonumber\\
&=& 
\left|c_1
+\sum\limits_{j=1}^k\Big(\bar{c}'_k(j)-\bar{c}_k(j)\Big)d_{k+1}(j)
+\sum\limits_{j=1}^k\bar{c}_k(j)\Big(d_{k+1}(j)-d_k(j)\Big)
+\bar{c}'_k(p)d_2
+(\bar{c}'_k-\bar{c}_k)m_k
-\bar{c}'_kd_2
\right|\label{eclaim11-1}\\
& \leq & 
d_1
+\frac{2}{n-k}\sum\limits_{j=1}^kd_{k+1}(j)
+\sum\limits_{j=1}^k|\bar{c}_k(j)|\Big|d_{k+1}(j)-d_k(j)\Big|
+\Big(1+|\bar{c}'_k|\Big)d_2
+\frac{4}{n-k}m_k\label{eclaim11}
\end{eqnarray}
First, we consider the case that $k+1\leq \epsilon n$.

Trivially, 
$$n-k\geq (1-\epsilon)n\stackrel{\epsilon<\frac{1}{4}}{\geq}\frac{3}{4}n,\,\,\,\,\,\,\,\,
\sum\limits_{j=1}^kd_{k+1}(j)\leq m(F)<n,\,\,\,\,\mbox{ and }\,\,\,\,\,\,\,\,
m_k\leq m(F)<n.$$
By (\ref{ec1}),
we have 
$$\sum\limits_{j=1}^k\Big|d_{k+1}(j)-d_k(j)\Big|=d_1\leq 1.$$
Since $c$ is a zero-sum labeling,
a simple inductive argument based on Claim \ref{claim3} implies
$$|\bar{c}'_k|
\leq \frac{4}{n}+\frac{4}{n-1}+\cdots+\frac{4}{n-k}
\leq \frac{4(k+1)}{n-k}
\leq \frac{4\epsilon n}{n-k}
\leq \frac{16}{3}\epsilon.$$
Now, (\ref{eclaim11}) implies
\begin{eqnarray*}
&&|\mathbb{E}[i_1,\ldots,i_k,p]-\mathbb{E}[i_1,\ldots,i_k]|\\
& \leq & 
\underbrace{d_1}_{\leq 1}
+\underbrace{\frac{2}{n-k}}_{\leq \frac{2}{\frac{3}{4}n}}
\underbrace{\sum\limits_{j=1}^kd_{k+1}(j)}_{\leq n}
+\underbrace{\sum\limits_{j=1}^k\underbrace{|\bar{c}_k(j)|}_{\leq 1}\Big|d_{k+1}(j)-d_k(j)\Big|}_{\leq 1}
+\Big(1+\underbrace{|\bar{c}'_k|}_{\leq \frac{16}{3}\epsilon}\Big)\underbrace{d_2}_{\leq \Delta}
+\underbrace{\frac{4}{n-k}}_{\leq\frac{4}{\frac{3}{4}n}}
\underbrace{m_k}_{\leq n}\\
& \leq & 
1
+\frac{2n}{\frac{3}{4}n}
+1
+\left(1+\frac{16}{3}\epsilon\right)\Delta
+\frac{4n}{\frac{3}{4}n}\\
&=&\left(1+\frac{16}{3}\epsilon\right)\Delta+10\\
&\leq &\left(1+\frac{16}{3}\epsilon\right)\Delta+\left(\frac{8}{\epsilon}+4\right).
\end{eqnarray*}
Next, we consider the case that $k+1>\epsilon n$.

By (\ref{ec2}), we obtain that $d_1+d_2=d_F(k+1)\leq \frac{2}{\epsilon}$,
and also that 
$$\sum\limits_{j=1}^kd_{k+1}(j)
\leq \sum\limits_{j=k+1}^nd_F(j)
\leq \frac{2}{\epsilon}(n-k).$$
Since $F$ is a forest, we have $m_k=m(F-[k])<n-k$.

Now, (\ref{eclaim11}) implies
\begin{eqnarray}
&&|\mathbb{E}[i_1,\ldots,i_k,p]-\mathbb{E}[i_1,\ldots,i_k]|\nonumber\\
& \leq & 
d_1
+\frac{2}{n-k}\underbrace{\sum\limits_{j=1}^kd_{k+1}(j)}_{\leq \frac{2}{\epsilon}(n-k)}
+\underbrace{\sum\limits_{j=1}^k\underbrace{|\bar{c}_k(j)|}_{\leq 1}\Big|d_{k+1}(j)-d_k(j)\Big|}_{\leq d_1}
+\Big(1+\underbrace{|\bar{c}'_k|}_{\leq 1}\Big)d_2
+\frac{4}{n-k}\underbrace{m_k}_{\leq n-k}\nonumber\\
& \leq & \underbrace{2d_1+2d_2}_{\leq \frac{4}{\epsilon}}+\frac{4}{\epsilon}+4\nonumber\\
& \leq & \frac{8}{\epsilon}+4\label{eclaim1r}\\
&\leq &\left(1+\frac{16}{3}\epsilon\right)\Delta+\left(\frac{8}{\epsilon}+4\right),\nonumber
\end{eqnarray}
which completes the proof.
\end{proof}
For the proof of Claim \ref{claim2},
we need the following generalization of Theorem 2 from~\cite{mopara}.

\begin{claim}\label{claim4}
Let the positive real $\epsilon$ and the integer $n$ 
be such that $\epsilon n\geq 10$.

If $G$ is a graph of order $n$ and size $m$ such that 
$$\left|m-\frac{1}{2}{n\choose 2}\right|\leq \frac{\epsilon}{10}{n\choose 2},$$
then 
$$\left(\frac{1}{4}-\epsilon\right)n\leq d_G(u)\leq \left(\frac{3}{4}+\epsilon\right)n-1$$
for some vertex $u$ of $G$.
\end{claim}
\begin{proof}
Since the statement is trivial for $\epsilon\geq \frac{1}{4}$,
we may assume that $\epsilon<\frac{1}{4}$.
For a contradiction, 
suppose that $G$ is as in the hypothesis 
but that, for every vertex $u$ of $G$,
either $d_G(u)<\left(\frac{1}{4}-\epsilon\right)n$
or $d_G(u)>\left(\frac{3}{4}+\epsilon\right)n-1$.
Let 
$V_+=\{ u\in V(G):d_G(u)>\left(\frac{3}{4}+\epsilon\right)n-1\}$
and $V_-=V(G)\setminus V_+$.
Since $\left(\frac{3}{4}+\epsilon\right)n-1>\left(\frac{1}{4}-\epsilon\right)n$, we have
$V_-=\{ u\in V(G):d_G(u)<\left(\frac{1}{4}-\epsilon\right)n\}.$
We assume that among all counterexamples,
the graph $G$ is chosen such that 
$$\sum\limits_{u\in V_+}d_G(u)={n\choose 2}-\sum\limits_{u\in V_-}d_G(u)$$
is as large as possible.
Note that the desired statement as well as the choice of $G$
are symmetric with respect to forming the complement.

Let $n_+=|V_+|$. 
Clearly, $|V_-|=n-n_+$.
If $n_+>\frac{3}{4}n$, then 
$$
\frac{21}{40}n^2
>\frac{21}{20}{n\choose 2}
>\left(1+\frac{\epsilon}{5}\right){n\choose 2}
\geq 2m
\geq\sum\limits_{u\in V_+}d_G(u)
>\frac{3}{4}n\left(\left(\frac{3}{4}+\epsilon\right)n-1\right)
\stackrel{\epsilon n\geq 1}{\geq}\frac{9}{16}n^2,$$
which is a contradiction.
We obtain $n_+\leq \frac{3}{4}n$.
By symmetry with respect to forming the complement,
we also obtain $n-n_+\leq \frac{3}{4}n$,
and, hence, 
\begin{eqnarray}\label{e1}
\frac{1}{4}n\leq n_+\leq \frac{3}{4}n,
\end{eqnarray}
which implies, in particular, that every vertex in $V_+$ has a neighbor in $V_-$.

If $V_+$ contains two non-adjacent vertices $u$ and $v$, 
and $w$ is a neighbor of $u$ in $V_-$,
then $G'=G-uw+uv$ is a counterexample 
contradicting the choice of $G$.
Hence, 
$V_+$ is complete.
By symmetry with respect to forming the complement,
the choice of $G$ also implies that  
$V_-$ is independent.

Now, let $d_+$ be the average degree of the vertices in $V_+$,
and let $d_-$ be the average degree of the vertices in $V_-$.
Clearly, 
\begin{eqnarray}\label{e1b}
d_+>\left(\frac{3}{4}+\epsilon\right)n-1\,\,\,\,\,\,\,\mbox{ and }\,\,\,\,\,\,\,
d_-<\left(\frac{1}{4}-\epsilon\right)n.
\end{eqnarray}
Let the real $n'$ be such that 
$$2m={n'\choose 2}.$$
If $n'>\left(1+\frac{4}{3}\epsilon\right)n$, then 
\begin{eqnarray*}
\left|m-\frac{1}{2}{n\choose 2}\right|
&=&\left|\frac{1}{2}{n'\choose 2}-\frac{1}{2}{n\choose 2}\right|\\
&>&\frac{1}{2}{\left(1+\frac{4}{3}\epsilon\right)n\choose 2}-\frac{1}{2}{n\choose 2}\\
&=&\frac{\epsilon n(6n+4\epsilon n-3)}{9}\\
&\stackrel{\epsilon n\geq 1}{>}&\frac{6\epsilon n^2}{9}\\
&>&\frac{\epsilon}{10}{n\choose 2},
\end{eqnarray*}
which is a contradiction.
Conversely, 
if $n'<\left(1-\frac{4}{5}\epsilon\right)n$, then 
\begin{eqnarray*}
\left|m-\frac{1}{2}{n\choose 2}\right|
&=&\left|\frac{1}{2}{n'\choose 2}-\frac{1}{2}{n\choose 2}\right|\\
&>&\frac{1}{2}{n\choose 2}-\frac{1}{2}{\left(1-\frac{4}{5}\epsilon\right)n\choose 2}\\
&=&\frac{\epsilon n(10n-4\epsilon n-5)}{25}\\
&\stackrel{\epsilon n\geq 1}{>}&\frac{\epsilon (10-9\epsilon)n^2}{25}\\
&>&\frac{\epsilon}{10}{n\choose 2},
\end{eqnarray*}
which is a contradiction.
Hence,
\begin{eqnarray}\label{e2}
\left(1-\frac{4}{5}\epsilon\right)n\leq n'\leq \left(1+\frac{4}{3}\epsilon\right)n.
\end{eqnarray}
Note that 
\begin{eqnarray*}
\frac{n'-n_+}{n-n^+}
&\stackrel{(\ref{e1}),(\ref{e2})}{\geq} &\frac{\left(1-\frac{4}{5}\epsilon\right)n-\left(1-\frac{4}{5}\epsilon\right)n_+-\frac{4}{5}\epsilon n_+}{n-n^+}\\
&=&\left(1-\frac{4}{5}\epsilon\right)-\frac{4}{5}\epsilon\left(\frac{n_+}{n-n^+}\right)\\
&\stackrel{(\ref{e1})}{\geq} & 1-\frac{4}{5}\epsilon-\frac{4}{5}\epsilon\left(\frac{\frac{3}{4}n}{\frac{1}{4}n}\right)\\
&=&1-\frac{16}{5}\epsilon,
\end{eqnarray*}
which implies
\begin{eqnarray}
\frac{1}{2}n'(n'-n_+)
& \stackrel{(\ref{e2})}{\geq} & \frac{1}{2}\left(1-\frac{4}{5}\epsilon\right)n\left(1-\frac{16}{5}\epsilon\right)(n-n_+)\nonumber\\
&>&\frac{1}{2}\left(1-4\epsilon\right)n(n-n_+)\nonumber\\
&=&2\left(\frac{1}{4}-\epsilon\right)n(n-n_+).\label{e2b}
\end{eqnarray}
Since every vertex in $V_+$ has exactly $n_+-1$ neighbors in $V_+$,
and $V_-$ is independent,
the number of edges in $G$ between $V_+$ and $V_-$ equals
\begin{eqnarray}\label{e3}
d_+n_+-n_+(n_+-1)=d_-(n-n_+),
\end{eqnarray}
and the sum of all vertex degrees equals
\begin{eqnarray}\label{e4}
d_+n_++d_-(n-n_+)=2m={n'\choose 2}.
\end{eqnarray}
Adding (\ref{e3}) and (\ref{e4}) implies
$$\frac{1}{2}n'(n'-1)+n_+(n_+-1)=2d_+n_+
\stackrel{(\ref{e1b})}{>}
2\left(\left(\frac{3}{4}+\epsilon\right)n-1\right)n_+
\stackrel{(\ref{e2})}{\geq}2\left(\frac{3}{4}n'-1\right)n_+$$
or, equivalently,
$$\left(n_+-\frac{n'}{2}\right)\left(n_+-(n'-1)\right)>0.$$
Since 
$n_+
\stackrel{(\ref{e1})}{\leq}\frac{3}{4}n
\stackrel{\epsilon<\frac{1}{4},\epsilon n\geq 10}{<}\left(1-\frac{4}{5}\epsilon\right)n-1
\stackrel{(\ref{e2})}{\leq}n'-1,$
this implies 
$$n_+<\frac{n'}{2}.$$
Subtracting (\ref{e3}) from (\ref{e4}) implies
$$\frac{1}{2}n'(n'-1)-n_+(n_+-1)=2d_-(n-n_+)
\stackrel{(\ref{e1b})}{<}
2\left(\frac{1}{4}-\epsilon\right)n(n-n_+)
\stackrel{(\ref{e2b})}{\leq}\frac{1}{2}n'(n'-n_+)$$
or, equivalently,
$$\left(n_+-\frac{n'}{2}\right)\left(n_+-1\right)>0.$$
Since $\epsilon <\frac{1}{4}$, $\epsilon n\geq 10$, and (\ref{e1}), 
it follows easily that $n_+>1$, and, hence, 
$$n_+>\frac{n'}{2}.$$
The contradiction $\frac{n'}{2}<n_+<\frac{n'}{2}$ completes the proof.
\end{proof}
We are now in the position to complete the proof of Claim \ref{claim2}.
\begin{proof}[Proof of Claim \ref{claim2}]
If $k+1>\epsilon n$, 
then exactly the same arguments as in the proof of Claim \ref{claim1}
imply that 
$$|\mathbb{E}[i_1,\ldots,i_k,p]-\mathbb{E}[i_1,\ldots,i_k]|
\stackrel{(\ref{eclaim1r})}{\leq}\frac{8}{\epsilon}+4
\leq 
\left(\frac{1}{2}+327\epsilon\right)\Delta+\left(\frac{8}{\epsilon}+4\right),$$
regardless of the specific choice of $p$ from $[n]\setminus \{ i_1,\ldots,i_k\}$.

Hence, we may assume that $k+1\leq \epsilon n$.

We consider the auxiliary graph 
$$G=\left([n]\setminus [k],c^{-1}(1)\cap {[n]\setminus [k]\choose 2}\right).$$
Since $c$ is a zero-sum labeling, 
we have $|c^{-1}(1)|=\frac{1}{2}{n\choose 2}$.

Note that 
\begin{eqnarray}
\frac{n^2}{{n-k\choose 2}}
\leq
\frac{n^2}{{(1-\epsilon)n\choose 2}}
=
\frac{2n}{(1-\epsilon)((1-\epsilon)n-1)}
\stackrel{\epsilon<\frac{1}{4}}{\leq}
\frac{2n}{\frac{3}{4}\left(\frac{3}{4}n-1\right)}
\leq \frac{16}{3},\label{eclaim22}
\end{eqnarray}
where the last inequality assumes that $n$ is large enough to ensure 
$\frac{3}{4}n-1\geq \frac{n}{2}$.

Since the graph $G$ is obtained from the graph with vertex set $[n]$
and edge set $c^{-1}(1)$ by removing $k$ vertices, we obtain
\begin{eqnarray*}
\left|m(G)-\frac{1}{2}{n-k\choose 2}\right|
& \leq & 
\left|m(G)-\frac{1}{2}{n\choose 2}\right|
+
\left|\frac{1}{2}{n\choose 2}-\frac{1}{2}{n-k\choose 2}\right|\\
& \leq & \Big((n-1)+(n-2)+\cdots+(n-k)\Big)+
\underbrace{\left|\frac{1}{2}{n\choose 2}-\frac{1}{2}{(1-\epsilon)n\choose 2}\right|}_{=\frac{\epsilon n((2-\epsilon)n-1)}{4}\leq \epsilon n^2}\\
& \leq & kn+\epsilon n^2\\
& \leq & 2\epsilon n^2\\
& \stackrel{(\ref{eclaim22})}{\leq}&\frac{32\epsilon}{3}{n-k\choose 2}.
\end{eqnarray*}
Now, applying Claim \ref{claim4} 
with $\frac{320\epsilon}{3}$ instead of $\epsilon$,
implies the existence of a vertex $p$ of $G$
with 
$$\left(\frac{1}{4}-\frac{320\epsilon}{3}\right)(n-k)\leq d_G(p)\leq \left(\frac{3}{4}+\frac{320\epsilon}{3}\right)(n-k)-1.$$
Using the same notation as in the proof of Claim \ref{claim1},
this implies
\begin{eqnarray*}
|\bar{c}'_k(p)|&\leq &
\frac{1}{n-k-1}\left(
\left(\frac{3}{4}+\frac{320\epsilon}{3}\right)(n-k)-1
-\left(\frac{1}{4}-\frac{320\epsilon}{3}\right)(n-k)\right)\\
&\stackrel{k+1\leq \epsilon n}{\leq} &\frac{1}{(1-\epsilon)n}
\left(\frac{1}{2}+\frac{640\epsilon}{3}\right)n\\
&\stackrel{\epsilon<\frac{1}{4}}{\leq}&
(1+2\epsilon)\left(\frac{1}{2}+\frac{640\epsilon}{3}\right)\\
&\stackrel{\epsilon<\frac{1}{4}}{\leq}&
\frac{1}{2}+\frac{963\epsilon}{3}.
\end{eqnarray*}
Now, 
using the notation and some observations from the proof of Claim \ref{claim1},
we have 
\begin{eqnarray*}
&& |\mathbb{E}[i_1,\ldots,i_k,p]-\mathbb{E}[i_1,\ldots,i_k]|\\
&\stackrel{(\ref{eclaim11-1})}{=}& 
\left|c_1
+\sum\limits_{j=1}^k\Big(\bar{c}'_k(j)-\bar{c}_k(j)\Big)d_{k+1}(j)
+\sum\limits_{j=1}^k\bar{c}_k(j)\Big(d_{k+1}(j)-d_k(j)\Big)
+\bar{c}'_k(p)d_2
+(\bar{c}'_k-\bar{c}_k)m_k
-\bar{c}'_kd_2
\right|\\
& \leq & 
\underbrace{d_1}_{\leq 1}
+\underbrace{\frac{2}{n-k}\sum\limits_{j=1}^kd_{k+1}(j)}_{\leq \frac{2}{\frac{3}{4}n}n\leq \frac{8}{3}}
+\underbrace{\sum\limits_{j=1}^k|\bar{c}_k(j)|\Big|d_{k+1}(j)-d_k(j)\Big|}_{\leq 1}
+\underbrace{\Big(|\bar{c}'_k(p)|+|\bar{c}'_k|\Big)}_{\leq \frac{1}{2}+\frac{963\epsilon}{3}+\frac{16}{3}\epsilon\leq \frac{1}{2}+327\epsilon}
\underbrace{d_2}_{\leq \Delta}
+\underbrace{\frac{4}{n-k}m_k}_{\leq \frac{4}{\frac{3}{4}n}n\leq \frac{16}{3}}\\
& \leq & 
\left(\frac{1}{2}+327\epsilon\right)\Delta+\left(\frac{8}{\epsilon}+4\right),
\end{eqnarray*}
which completes the proof.
\end{proof}
As explained after the statement of Claim \ref{claim2},
this completes the proof of Theorem \ref{theorem1}

\section{Conclusion}

Unlike the proof of Proposition \ref{proposition2},
the proof of Theorem \ref{theorem1} uses that $F$ is a forest.
Nevertheless, it is not difficult to generalize the approach to $k$-degenerate graphs,
adapting, in particular, (\ref{ec1}) and (\ref{ec2})
as well as the estimates based thereupon.
Another possible and technically straightforward generalization 
concerns the situation in which $c$ is not zero-sum,
that is, $c(K)$ is not $0$.
In this case, our approach yields an isomorphic copy $F'$ of $F$
for which $\bar{c}(F')$ is close to $\bar{c}(K)$.
It seems possible to strengthen our approach, 
or rather the analysis of the considered greedy algorithm,
in order to get an approximate version of Conjecture \ref{conjecture2}
with $\frac{1}{2}\Delta+\frac{1}{2}$ 
replaced by $\left(\frac{1}{2}+\epsilon\right)\Delta+C_\epsilon$.
A key ingredient that needs to be improved for this 
seems to be Claim \ref{claim4}.

The transpositions considered in the proof of Proposition \ref{proposition2} 
suggest {\it local search} as another algorithmic strategy:
Let $K$ be a complete graph with vertex set $[n]$,
let $c:E(K)\to \{ -1,1\}$ be a zero-sum labeling of $K$, and 
let $H$ be a spanning subgraph of $K$.
For two distinct vertives $u$ and $v$ of $K$,
let $\delta_H(uv)$ be the set of edges of $H$ 
between $\{ u,v\}$
and $(N_H(u)\cup N_H(v))\setminus \{ u,v\}$,
and let $H_{uv}$ arise from $H$ by 
\begin{itemize}
\item removing all edges in $\delta_H(uv)$, and 
\item adding all possible edges between $u$ and $N_H(v)\setminus \{ u\}$
as well as between $v$ and $N_H(u)\setminus \{ v\}$,
\end{itemize}
that is, $H_{uv}$ is an isomorphic copy of $H$ in $K$ 
in which $u$ and $v$ exchanged their roles.

Now, we suppose that $c(H)>0$ but that 
$c(H)\leq c(H_{uv})$ for every edge $uv$ of $K$,
or, equivalently,
\begin{eqnarray}\label{elocal}
c(\delta_H(uv))\leq c(\delta_{H_{uv}}(uv))
\,\,\,\,\mbox{for every edge $uv$ of $K$,}
\end{eqnarray}
that is, no {\it local search step} replacing $H$ by $H_{uv}$ reduces $c(H)$.
Standard double-counting arguments imply
\begin{eqnarray*}
\sum\limits_{uv\in {[n]\choose 2}}c(\delta_H(uv)) & = & \sum\limits_{uv\in E(H)}(2n-4)c(uv),\\
\sum\limits_{uv\in {[n]\choose 2}}c(\delta_{H_{uv}}(uv))
&=& \sum\limits_{uv\in {[n]\choose 2}\setminus E(H)}(d_H(u)+d_H(v))c(uv)
+\sum\limits_{uv\in E(H)}(d_H(u)+d_H(v)-2)c(uv)\mbox{, and},
\end{eqnarray*}
summing (\ref{elocal}) over all edges $uv$ of $K$, we obtain
$$(2n-2)c(H)=\sum\limits_{uv\in E(H)}(2n-2)c(uv)\leq 
\sum\limits_{uv\in {[n]\choose 2}}(d_H(u)+d_H(v))c(uv).$$
The problem now is that the right hand side of this inequality is hard to work with.
However, if $H$ is $\Delta$-regular, then, since $c$ is zero-sum,
the right hand side evaluates to $0$,
which implies the contradiction $c(H)\leq 0$.
In other words, in the case that $H$ is $\Delta$-regular and $c(H)>0$,
there is at least one edge $uv$ of $K$ with $c(H_{uv})<c(H)$.
Therefore, 
since $|c(H_{uv})-c(H)|\leq 4\Delta$ for every edge $uv$ of $K$,
local search efficiently generates 
an isomorphic copy $H'$ of $H$ in $K$ with $|c(H')|\leq 2\Delta$.

\end{document}